\documentclass[11pt]{amsart}
\usepackage{tabularx,booktabs}
\usepackage{caption}
\usepackage{amsmath}
\usepackage{amsfonts}
\usepackage{bbold}

\usepackage{amscd}
\usepackage{amsthm}
\usepackage{amssymb} \usepackage{latexsym}
\usepackage{eufrak}
\usepackage{euscript}
\usepackage{epsfig}
\usepackage{graphics}
\usepackage{array}
\usepackage{enumerate}
\usepackage{dsfont}
\usepackage{color}
\usepackage{wasysym}
\usepackage{hyperref}
\usepackage{pdfsync}

\usepackage{graphicx}
\usepackage{float}
\usepackage{subfigure}

\usepackage{tikz}

\newcommand{\bel}[1]{\begin{equation}\label{#1}}

\newcommand{\be}{\begin{equation}}

\newcommand{\ba}{\begin{eqnarray}}
\newcommand{\ea}{\end{eqnarray}}

\newcommand{\qe}{\end{equation}}

\newcommand{\Hmm}[1]{\leavevmode{\marginpar{\tiny%
$\hbox to 0mm{\hspace*{-0.5mm}$\leftarrow$\hss}%
\vcenter{\vrule depth 0.1mm height 0.1mm width \the\marginparwidth}%
\hbox to
0mm{\hss$\rightarrow$\hspace*{-0.5mm}}$\\\relax\raggedright #1}}}

\newtheorem{theorem}{Theorem}[section]

\newtheorem{definition}[theorem]{Definition}

\newtheorem{remark}[theorem]{Remark}

\newtheorem{proposition}[theorem]{Proposition}

\begin{document}

\title[A version of Bakry-\'{E}mery Ricci flow on a finite graph]{A version of Bakry-\'{E}mery Ricci flow on a finite graph}

\author{Bobo Hua}
\address{Bobo Hua: School of Mathematical Sciences, LMNS, Fudan University, Shanghai 200433, China; Shanghai Center for Mathematical Sciences, Fudan University, Shanghai 200433, China}
\email{\href{mailto:bobohua@fudan.edu.cn}{bobohua@fudan.edu.cn}}

\author{Yong Lin}
\address{Yong Lin: Yau Mathematical Science Center, Tsinghua University, Beijing 100084, China}
\email{\href{mailto:yonglin@tsinghua.edu.cn}{yonglin@tsinghua.edu.cn}}

\author{Tao Wang}
\address{Tao Wang: School of Mathematical Sciences, Fudan University, Shanghai 200433, China}
\email{\href{mailto:taowang21@m.fudan.edu.cn}{taowang21@m.fudan.edu.cn}}

\maketitle
\bigskip

\begin{abstract}
In this paper, we study the Bakry-\'{E}mery Ricci flow on finite graphs. Our main result is the local existence and uniqueness of solutions to the Ricci flow. We prove the long-time convergence or finite-time blow up for the Bakry-\'{E}mery Ricci flow on finite trees and circles.
\end{abstract}

\section{Introduction}

An important problem in Riemannian geometry is to study the Ricci curvature of manifolds. The Ricci curvature of a manifold reflects the topological, geometric, and analytical properties of the manifold. In 1982, Richard Hamilton \cite{Hamilton1982} introduced the Ricci flow, which is a process that deforms the metric of a Riemannian manifold in a way formally analogous to the heat diffusion. The Ricci flow is a powerful tool for studying manifolds. The Ricci flow leads to the proofs of the Poincar\'{e} conjecture, Thurston's geometrization conjecture and the differentiable sphere theorem. We refer to \cite{perelman2002entropy, BS2009, CZ2006} and their references for more details.

In recent years, many mathematicians are interested in the analysis on graphs. Analogous to the continuous case, various Ricci curvatures were introduced on graphs. One is Ollivier's coarse Ricci curvature. In 2009, Ollivier \cite{Ollivier2009} used random walks to define a coarse Ricci curvature on a general metric space, and proved that on Riemannian manifolds, the coarse Ricci curvature coincides with the Ricci curvature in small scale. In particular, Ollivier's coarse Ricci curvature can be defined on graphs. In \cite{LLY2011}, the authors modified the definition of Ollivier's coarse Ricci curvature and defined the curvature that we now call Lin-Lu-Yau Ricci curvature. This curvature can be characterized by the Laplacian on the graph. We refer to \cite{MW2019}. Another curvature notion on the graph is the Bakry-\'{E}mery curvature. The definition of Bakry-\'{E}mery curvature comes from the Bochner formula in geometry, see Section 2 for a brief introduction.

Motivated by the continuous Ricci flow, there are some definitions of Ricci flow on graphs. Ni et al introduced a coarse Ricci flow on complex networks \cite{NLLG2019}. Lin-Lu-Yau Ricci flow was introduced by Bai et al. on finite graphs \cite{bai2021ollivier}. Moreover, they proved the long-time existence of the flow up to some surgeries. In \cite{cushing2022bakryemery, cushing2022bakryemery1}, the authors introduced a curvature on weighted graphs which is based on the Bakry-\'{E}mery curvature, and proved that the flow preserves the Markovian property and its limits as time goes to infinity turn out to be curvature sharp weighted graphs.

In this paper, we introduce a Ricci flow based on the Bakry-\'{E}mery Ricci curvature on finite graphs. Let $(V, E, w, m)$ be a finite, undirected, simple graph with the set of vertices $V$, the set of edges $E$, the edge weight $w: E \to \mathbb{R}_+$ and the vertex weight $m: V \to \mathbb{R}_+$. We fix the edge weight $w$ and vary the vertex weight $m$. Noting that the minimal eigenvalue of a symmetric matrix is locally Lipschitz continuous w.r.t. the matrix entries, we can prove the following local existence and uniqueness theorem of the Bakry-\'{E}mery Ricci flow.
\begin{theorem}
 For a finite graph $(V, E, w)$, and any $m_0: V \to \mathbb{R}_+$, there exist $T >0$ and a unique solution $m \in C^{\infty}([0, T) \times V, \mathbb{R}_+)$ to the Bakry-\'{E}mery Ricci flow
 \begin{equation*}
 \begin{cases}
  \partial_t m(t, x) = -\mathrm{Ric}_{n, m(t, \cdot)}(x), \quad \forall \ x \in V, \\
  m(0, x) = m_0(x), \quad \forall \ x \in V.
 \end{cases}
\end{equation*} 
where $\mathrm{Ric}_{n, m(t, \cdot)}(x)$ is the Bakry-\'{E}mery Ricci curvature defined in Definition \ref{BERicciCurvature}
\end{theorem}
We can also define the normalized Bakry-\'{E}mery Ricci flow 
\begin{align*}
\begin{cases}
 \partial_t m(t, x) = -\mathrm{Ric}_{n, m(t, \cdot)}(x) + \frac{1}{|V|}\sum_{y \in V}\mathrm{Ric}_{n, m(t, \cdot)}(y), \quad \forall \ x \in V, \\
 m(0, x) = m_0(x), \quad \forall \ x\in V,
\end{cases}
\end{align*}
which preserves the volume of the graph, and prove the existence and uniqueness of the normalized Ricci flow.

As examples, we prove that the Bakry-\'{E}mery Ricci flow on finite trees $T_3$, circle $C_3$ and $C_4$ blows up in finite time, and the flow on $C_k$ with $k \geq 5$ converges to a constant.

This paper is organized as follows. In Section 2, we recall the setting of weighted graphs and introduce the Bakry-\'{E}mery Ricci flow. Moreover, we will prove the main result. In Section 3, we discuss the Ricci flow on finite trees $T_3$ and circle $C_k$. 

\section{The Bakry-\'{E}mery Ricci flow}
\subsection{Weighted graphs}
In this subsection, we recall the setting of weighted graphs. Let $(V, E)$ be a finite, simple, undirected graph. Two vertices $x, y$ are called neighbors, denoted by $x \sim y$, if there is an edge connecting $x$ and $y$, i.e., $\{x, y\} \in E$. A graph is called connected if for any $x, y \in V$ there is a path $\{z_i\}_{i=0}^n \subseteq V$ connecting $x$ and $y$, i.e.,
\[
 x = z_0 \sim z_1 \sim \cdots \sim z_n = y.
\]
In this paper, we always consider connected graphs. Let
\[
 w: E \to \mathbb{R}_+, \quad \{x, y\} \mapsto w(x, y) = w(y, x),
\]
be an edge weight function, and
\[
 m: V \to \mathbb{R}_+, \quad x \mapsto m(x),
\]
be a vertex function.We write
\[
 m(V):= \sum_{x \in V}m(x)
\]
for the volume of the graph. We call the quadruple $G = (V, E, w, m)$ a \textit{weighted graph} with the edge weight $w$ and the vertex weight $m$. For a weighted graph $G$ and any function $f: V \to \mathbb{R}$, the Laplace operator $\Delta$ is defined as 
\[
 \Delta f(x):= \sum_{y \in V: y \sim x} \frac{w(x, y)}{m(x)}(f(y)-f(x)), \quad \forall \ x \in V.
\]

\subsection{Bakry-\'{E}mery curvature}
The Bochner formula is a fundamental tool to study the analysis of a manifold with Ricci curvature bound. Let $K \in \mathbb{R}$ and $n \in (0, \infty]$. As is well-known, for a Riemannian manifold $M$, the Ricci curvature is bounded below by $K$ and the dimension is bounded above by $n$, if and only if
\begin{equation}\label{BochnerFormula}
 \frac12\Delta_{M}|\nabla f|^2 \geq \frac{1}{n}(\Delta f)^2 + \langle \nabla f, \nabla \Delta f \rangle + K|\nabla f|^2, \quad \forall \ f \in C_c^{\infty}(M),
\end{equation}
where $\Delta_M$ is the Laplace-Beltrami operator on $M$ and $\nabla \cdot$ is the gradient of a function. For a general Markov semigroup, Bakry and \'{E}mery \cite{Bakry1987, BakryEmery1985, BGL2014} introduced the $\Gamma$-calculus, and defined the curvature dimension condition mimicking \eqref{BochnerFormula}, denoted by $CD(K, n)$. For weighted graphs, this condition is called the Bakry-\'{E}mery curvature condition, introduced by \cite{Elworthy1991, LinYau2010, Sch1999} independently.

Now we study the $\Gamma$-calculus on graphs, see \cite{Sch1999, LinYau2010, BGL2014}. For a weighted graph $G=(V, E, w, m)$, we introduce the gradient form (called Carr\'{e} du champ operator) as, for $f, g: V \to \mathbb{R}$,
\begin{equation}\label{GammaOperator}
 \Gamma(f, g):= \frac12\left(\Delta(fg)-f\Delta g- g\Delta f\right), \quad \Gamma(f):= \Gamma(f, f).
\end{equation}
For a weighted graph $G$, the term $|\nabla f|^2$ in \eqref{BochnerFormula} is understood as $\Gamma(f)$. The iterated gradient form is defined as 
\[
 \Gamma_2(f):= \frac12\Delta\Gamma(f) - \Gamma(\Delta f, f).
\]
\begin{definition}\label{BERicciCurvature}(\cite{Sch1999, LinYau2010})
 Let $K \in \mathbb{R}$, $n \in (0, \infty]$. For a vertex $x$, we say a weighted graph $G = (V, E, w, m)$ satisfies $CD(K, n, x)$ if 
 \[
  \Gamma_2(f)(x) \geq \frac{1}{n}(\Delta f)^2(x) + K\Gamma(f)(x), \quad \forall \ f: V \to \mathbb{R}.
 \]
 For any $x \in V$, $n \in (0, \infty]$, we denote by $\mathrm{Ric}_{n, w, m}(x)$ the maximal $K$ such that the above inequality holds for all $f$. If $CD(K, n, x)$ holds for all $x \in V$, we write $CD(K, n)$.
\end{definition}
\begin{remark}
 \begin{itemize}
  \item[(1)] The curvature condition $CD(K, n, x)$ is a local condition, which is determined by the graph structure in $B_2(x)$.
  \item[(2)] This is called linear curvature condition, since it is equivalent to the smallest eigenvalue problem for a quadratic form, see \cite{CKLP2022}. $\mathrm{Ric}_{n, w, m}(x)$ can be calculated by using linear programming \cite{CKLLS2022}.
 \end{itemize}
\end{remark}

\subsection{Bakry-\'{E}mery Ricci Flow} 
From now on, we fix the edge weight $w$, and omit the dependence on $w$. We call $(V, E, w)$ the reference graph. For a reference graph, we write $\mathrm{Ric}_{n, m}(x):= \mathrm{Ric}_{n, w, m}(x)$. If $n=\infty$, we write $\mathrm{Ric}_{m}(x):= \mathrm{Ric}_{\infty, m}(x)$.

We introduce the Bakry-\'{E}mery Ricci flow on a reference graph $(V, E, w)$. Let $m(t, \cdot)$ be a time-dependent vertex weight on $(V, E)$, i.e., $m: [0, T] \times V \to \mathbb{R}_{+}$ for $T > 0$. For $n \in (0, \infty]$, the Bakry-\'{E}mery Ricci flow of dimension $n$ is defined as 
\begin{equation}\label{RicciFlow}
 \begin{cases}
  \partial_t m(t, x) = -\mathrm{Ric}_{n, m(t, \cdot)}(x), \quad \forall \ x \in V, t \in [0,T], \\
  m(0, x) = m_0(x), \quad \forall \ x \in V,
 \end{cases}
\end{equation}
where $m_0$ is an initial vertex weight, and $\mathrm{Ric}_{n, m(t, \cdot)}$ is the Bakry-\'{E}mery curvature w.r.t. the vertex weight $m(t, \cdot)$.

By a reformation of the Bakry-\'{E}mery curvature, $\mathrm{Ric}_{n, m}(x)$ is the minimal eigenvalue of the curvature matrix $A_{n,m}(x)$ at $x$, see \cite{CKLP2022}. Moreover, we note that the element of $A_{n, m}(x)$ is a smooth function w.r.t. the vertex weight $m$. The following property is well-known, see Tao's book on random matrices \cite{Tao2012}.

\begin{proposition}
 For the symmetric matrix $A = (a_{ij})$, denote by $\lambda_{\min}(A)$ the minimal eigenvalue of $A$. Then $\lambda_{\min}(A)$ is locally Lipschitz continuous w.r.t $a_{ij}$.
\end{proposition}

We prove the local existence and uniqueness of the Bakry-\'{E}mery Ricci flow.

\begin{theorem}\label{MainTheorem}
 For a finite reference graph $(V, E, w)$, and any $m_0: V \to \mathbb{R}_+$, there exist $T > 0$ and a unique solution $m \in C^{\infty}([0, T) \times V, \mathbb{R}_+)$ to the Bakry-\'{E}mery Ricci flow \eqref{RicciFlow}.
\end{theorem}
\begin{proof}
 Note that for any $m \in \mathbb{R}_+$, 
 \[
  \mathrm{Ric}_{n, m}(x) = \lambda_{\min}(A_{n, m}(x)).
 \]
 Hence, it is locally Lipschitz continuous w.r.t. $m$. By the Picard theorem, the result follows.
\end{proof}
Here $T$ is called the maximal existence time of the Bakry-\'{E}mery Ricci flow.

\subsection{Normalized Bakry-\'{E}mery Ricci Flow}
For a finite reference graph, we define the normalized Bakry-\'{E}mery Ricci flow, which preserves the volume of the graph, via
\begin{align*}
\begin{cases}
 \partial_t m(t, x) = -\mathrm{Ric}_{n, m(t, \cdot)}(x) + \frac{1}{|V|}\sum_{y \in V}\mathrm{Ric}_{n, m(t, \cdot)}(y), \quad \forall \ x \in V, t \in [0, T], \\
 m(0, x) = m_0(x), \quad \forall \ x \in V,
\end{cases}
\end{align*}
where $m_0$ is an initial vertex weight. Then we also have the local existence and uniqueness of the normalized Bakry-\'{E}mery Ricci flow by the same argument as in the proof of Theorem \ref{MainTheorem}. 
\begin{theorem}
 For a finite reference graph $(V, E, w)$, and any $m_0: V \to \mathbb{R}_+$, there exist $T > 0$ and a unique solution $m \in C^{\infty}([0, T) \times V, \mathbb{R}_+)$ to the normalized Bakry-\'{E}mery Ricci flow.
\end{theorem}

\section{Examples}
In this section, we give some examples of the Bakry-\'{E}mery Ricci flows. We always assume the trivial edge weight, i.e., $w \equiv 1$. The first example is the Bakry-\'{E}mery Ricci flow on a finite tree. Consider the following tree $T_3$.
\begin{figure}[H]
 \centering
  \begin{tikzpicture}
   \filldraw[draw = black, fill = black] (0,0) circle(0.05) (1,0) circle(0.05) (1.8,0.5) circle(0.05) (1.8,-0.5) circle(0.05) (2.75,0.2) circle(0.05) (2.75,0.8) circle(0.05);
   \filldraw[draw = black, fill = black] (2.75,-0.2) circle(0.05) (2.75,-0.8) circle(0.05) (-0.8,0.5) circle(0.05) (-0.8,-0.5) circle(0.05);
   \draw (-0.8,0.5) -- (0,0) -- (1,0) -- (1.8,0.5) -- (2.75,0.8);
   \draw (-0.8,-0.5) -- (0,0); \draw (1.8,0.5) -- (2.75,0.2); \draw(1,0) -- (1.8,-0.5) -- (2.75,-0.8); \draw (1.8,-0.5) -- (2.75,-0.2);
   \node at (1,0.3) {$x_1$}; \node at (0,0.3) {$x_2$}; \node at (1.6,0.7) {$x_3$}; \node at (1.6,-0.7) {$x_4$}; \node at (-1.1, 0.5) {$x_5$};
   \node at (-1.1,-0.5) {$x_6$}; \node at (3.05,0.8) {$x_7$}; \node at (3.05,0.2) {$x_8$}; \node at (3.05,-0.2) {$x_9$}; \node at (3.1,-0.8) {$x_{10}$};
  \end{tikzpicture}
  \caption{$T_3$ with boundary point  $\delta T_3 = \{x_5, x_6, \cdots, x_{10}\}$.}
  \label{T3}
\end{figure}
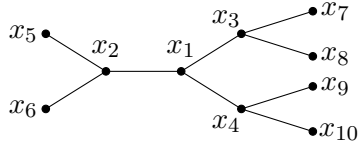
For any $m: V \to \mathbb{R}_+$, the Bakry-\'{E}mery curvatures are given by 
\begin{align*}
 &\mathrel{\phantom{=}}\mathrm{Ric}_m(x_1) \\
  &= \lambda_{\min}\left(\begin{bmatrix} \frac12(\frac{1}{m(x_2)}-\frac{1}{m(x_1)}) & \frac{1}{m(x_1)} &\frac{1}{m(x_1)} \\
 \frac{1}{m(x_1)} & \frac12(\frac{1}{m(x_3)}-\frac{1}{m(x_1)}) & \frac{1}{m(x_1)} \\ 
 \frac{1}{m(x_1)} & \frac{1}{m(x_1)} & \frac12(\frac{1}{m(x_4)}-\frac{1}{m(x_1)}) \end{bmatrix}\right),
\end{align*}
\begin{align*}
 &\mathrel{\phantom{=}}\mathrm{Ric}_m(x_2) \\
 &= \lambda_{\min}\left(\begin{bmatrix} \frac12(\frac{3}{m(x_5)}-\frac{1}{m(x_2)}) & \frac{1}{m(x_2)} & \frac{1}{m(x_2)} \\
 \frac{1}{m(x_2)} & \frac12(\frac{3}{m(x_6)}-\frac{1}{m(x_2)}) & \frac{1}{m(x_2)} \\
 \frac{1}{m(x_2)} & \frac{1}{m(x_2)} & \frac12(\frac{1}{m(x_1)}-\frac{1}{m(x_2)}) \end{bmatrix}\right),
\end{align*}
and
\[
 \mathrm{Ric}_m(x_5) = \frac12\left(\frac{1}{m(x_2)}+\frac{1}{m(x_5)}\right).
\]
We prove the finite-time blow up for the flow on the tree $T_3$.
\begin{theorem}
 For $T_3$ as in Figure \ref{T3}, $m_0: V \to \mathbb{R}_+$, let $m \in C^{\infty}([0, T) \times V, \mathbb{R}_+)$ be a solution to the Bakry-\'{E}mery Ricci flow. Then $$T \leq \frac{(\min_{x \in \delta T_3}m_0(x))^2}{2}.$$
\end{theorem}
\begin{proof}
 W.l.o.g., we assume that $m_0(x_5) = \min_{x \in \delta T_3}m_0(x)$. Since 
 \[
  \mathrm{Ric}_{m(t, \cdot)}(x_5) = \frac12\left(\frac{1}{m(t, x_2)} + \frac{1}{m(t, x_5)}\right)
 \]
and by the Ricci flow equation, we have
 \[
  \partial_tm(t, x_5) = -\mathrm{Ric}_{m(t, \cdot)}(x_5) \leq -\frac{1}{m(t, x_5)}.
 \]
 Hence, we get 
 \[
  m^2(t, x_5) \leq m_0^2(x_5) - 2t.
 \]
 It follows that $T \leq \frac{m_0^2(x_5)}{2}$, as desired.
\end{proof}

In the following, we consider the Bakry-\'{E}mery Ricci flow on circles. Let $C_k = (V, E)$ be the circle of length $k$, i.e., the set of vertices is given by $\{0, 1, \cdots, k-1\}$ and $i \sim j$ if and only if $|i-j|=1 \mod k$. We denote by $C_k$ the cycle of length $k$.

We firstly consider the Bakry-\'{E}mery Ricci flow on $C_3$. For $m: V \to \mathbb{R}_+$, the Bakry-\'{E}mery curvature is given by, for $i = 0, 1, 2$,
\begin{align*}
 &\mathrel{\phantom{=}}\mathrm{Ric}_m(i) \\
 &= -\sqrt{\left(\frac{1}{m(i)}-\frac{1}{m(i-1)}-\frac{1}{m(i+1)}\right)^2+\frac{25}{16}\left(\frac{1}{m(i-1)}-\frac{1}{m(i+1)}\right)^2} \\
 &\mathrel{\phantom{=}}+\frac{7}{4}\left(\frac{1}{m(i-1)}+\frac{1}{m(i+1)}\right). 
\end{align*}
\begin{proposition}\label{PreserveMaxAndMin}
 For $C_3$ and any $m_0: V \to \mathbb{R}_+$, let $m \in C^{\infty}([0, T) \times V, \mathbb{R}_+)$ be the solution to the Bakry-\'{E}mery Ricci Flow. 
 \begin{itemize}
  \item[(1).] If $m_0(0) > \max\{m_0(1), m_0(2)\}$, then
  \[
   m(t, 0) > \max\{m(t, 1), m(t, 2)\}, \quad \forall \ t \in [0, T).
  \]
  \item[(2).] If $m_0(0) < \min\{m_0(1), m_0(2)\}$, then
  \[
   m(t, 0) < \min\{m(t, 1), m(t, 2)\}, \quad \forall \ t \in [0, T).
  \]
 \end{itemize}
\end{proposition}
\begin{proof}
 W.l.o.g., we prove (1). Suppose that there exists $t_1 \in (0, T)$ such that $m(t_1, 0) = \max\{m(t_1, 1), m(t_1, 2)\}$. W.l.o.g., we assume that $m(t_1, 0) = m(t_1,1)$. Then we consider the following ODEs
 \begin{equation}\label{SameInitialData}
  \begin{cases}
   \partial_tm(t, 1) = -\frac74(\frac{1}{m(t, 1)}+\frac{1}{m(t, 2)})+\sqrt{\frac{1}{(m(t, 2))^2}+\frac{25}{16}(\frac{1}{m(t, 1)}-\frac{1}{m(t, 2)})^2}, \\
   \partial_tm(t, 2) = -\frac72\frac{1}{m(t, 1)}+\Big|\frac{1}{m(t, 2)}-\frac{2}{m(t, 1)}\Big|, \\
   m(t_1, 1) = m(t_1, 0), m(t_1, 2) = m(t_1, 2).
  \end{cases}
 \end{equation}
 According to the Picard theorem, there exist $\varepsilon > 0$ and a unique solution $(m(t,1), m(t, 2))$ on $(t_1-\varepsilon, t_1+\varepsilon)$ to \eqref{SameInitialData}. Then one can see that $(m(t, 1), m(t, 1), m(t, 2))$ is a solution to the Ricci flow on $(t_1-\varepsilon, t_1+\varepsilon)$. By the local uniqueness of the solution to ODEs, we get that
 \[
  m(t, 0) = m(t, 1), \quad \forall \ t \in (t_1-\varepsilon, t_1+\varepsilon).
 \]
 Repeating this process, we obtain that $\lim_{t \to 0}m(t, 0) = \lim_{t \to 0}m(t, 1)$, which is a contradiction to the $m_0(0) > m_0(1)$.
\end{proof}
\begin{remark}
 Due to the local uniqueness of the solution to ODEs, we know that for Bakry-\'{E}mery Ricci flow on $C_3$, if $m_0(0) = m_0(1)$, then for any $t \in [0, T)$, $m(t, 0) = m(t, 1)$.
\end{remark}
\begin{theorem}\label{MaximumDecreasingForC3}
 For $C_3$ and any $m_0: V \to \mathbb{R}_+$, let $m \in C^{\infty}([0, T) \times V, \mathbb{R}_+)$ be a solution to the Bakry-\'{E}mery Ricci flow. Then $\max_{x \in V}m(t, x)$ is decreasing in $t$. Moreover, $T \leq \max_{x \in V}(m_0(x))^2$.
\end{theorem}
\begin{proof}
 W.l.o.g., we assume that $m_0(0) \geq \max\{m_0(1), m_0(2)\}$, then we get that $m(t, 0) \geq \max\{m(t, 1), m(t, 2)\}$ for any $t \in [0, T)$. Noting that for any $m \in \mathbb{R}_+^V$ satisfying $m(0) \geq \max\{m(1), m(2)\}$, we have
 \begin{align*}
  &\mathrel{\phantom{=}}\mathrm{Ric}_m(0) \\
  &= \frac74(\frac{1}{m(1)}+\frac{1}{m(2)}) - \sqrt{(\frac{1}{m(0)}-\frac{1}{m(1)}-\frac{1}{m(2)})^2+\frac{25}{16}(\frac{1}{m(1)}-\frac{1}{m(2)})^2} \\
  &= \left(\frac12(\frac{1}{m(1)^2}+\frac{1}{m(2)^2})-\frac{1}{m(0)^2}+\frac{2}{m(0)m(1)}+\frac{2}{m(0)m(2)}+\frac{29}{4m(1)m(2)}\right) \\
  &\mathrel{\phantom{=}} \times (F(m))^{-1},
 \end{align*}
 where $F(m)$
 \begin{align*}
  F(m) &= \frac74(\frac{1}{m(1)}+\frac{1}{m(2)}) \\
  &\mathrel{\phantom{=}} + \sqrt{(\frac{1}{m(0)}-\frac{1}{m(1)}-\frac{1}{m(2)})^2+\frac{25}{16}(\frac{1}{m(1)}-\frac{1}{m(2)})^2} \\
  &\leq \frac74(\frac{1}{m(1)}+\frac{1}{m(2)}) + \frac54(\frac{1}{m(1)}+\frac{1}{m(2)}) + \frac{1}{m(1)}+\frac{1}{m(2)}-\frac{1}{m(0)} \\
  &\leq 4(\frac{1}{m(1)}+\frac{1}{m(2)}).
 \end{align*}
This yields that
 \[
  \mathrm{Ric}_m(0) \geq \frac{\frac{2}{m(0)m(1)}+\frac{2}{m(0)m(2)}}{4(\frac{1}{m(1)}+\frac{1}{m(2)})} = \frac{1}{2m(0)},
 \]
 and hence,
 \[
  \partial_tm(t, 0) \leq -\frac12\frac{1}{m(t, 0)},
 \]
 which implies that $m^2(t, 0) \leq m_0(0)^2 - t$, and $T \leq m_0(0)^2$. 
\end{proof}
\begin{remark}
 For $C_3$, unlike $\max_{x \in V}m(t, x)$, $\min_{x \in V}m(t, x)$ does not have monotonically increasing properties, i.e., $\min_{x \in V}m(t, x)$ may first increase and then decrease or decrease all the time. Another interesting observation is that for $i \in \{0, 1, 2\}$, $\mathrm{Ric}_{m(t, \cdot)}(i)$ are not equal to zero at the same time.
\end{remark}
\begin{proposition}
 For $C_3$ and any $m_0: V \to \mathbb{R}_+$ satisfying $m_0(0) < m_0(1) < m_0(2)$, let $m \in C^{\infty}([0, T) \times V, \mathbb{R}_+)$ be a solution to the Bakry-\'{E}mery Ricci flow. Suppose that $m(T, 0) = 0$, then $m(T, 1) = 0$.
\end{proposition}
\begin{proof}
 If not, we assume that $m(T, 1) = c_1 >0$, then $m(T, 2) = c_2 > c_1 > 0$, and hence, $\mathrm{Ric}_{m(t, \cdot)}(0) \to -\infty$, as $t \to T$, which implies that $\partial_t m(t, 0) \to \infty$ as $t \to T$.
  
On the other hand, since $m(t, 0) > 0$ for $t \in [0, T)$ and $m(T, 0) = 0$, we have $\partial_t m(T, 0) \leq 0$. The contradiction shows that $m(T, 1) =0$. As desired.
\end{proof}
\begin{figure}[H]
 \centering
  \includegraphics[width = 0.7\textwidth]{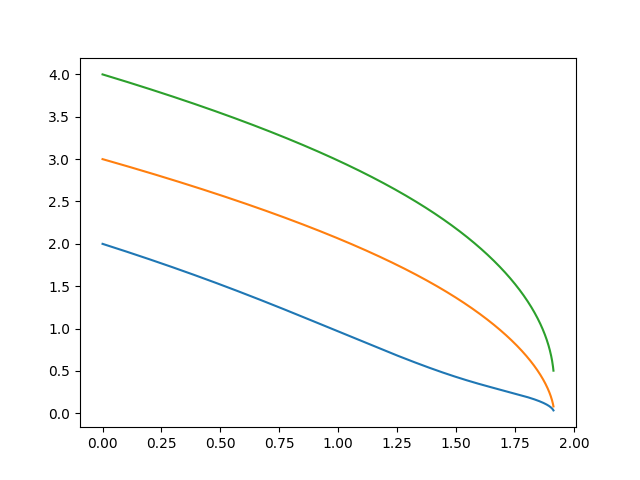}
 \caption{Bakry-\'{E}mery Ricci flow on $C_3$ with $m_0 = (2, 3, 4)$. In this cases, $\min_{x \in V}m(t, x)$ is decreasing.}
\end{figure}

Next we consider the Bakry-\'{E}mery Ricci flow on $C_4$. For $m: V \to \mathbb{R}_+$, the Bakry-\'{E}mery curvature is given by, for $i = 0, 1, 2, 3$,
\begin{align*}
 &\mathrel{\phantom{=}}\mathrm{Ric}_m(i) \\
 &=\frac{m(i-1)+m(i+1)}{2m(i-1)m(i+1)} + \frac{2}{m(i-1)+m(i+1)} \\
 &\mathrel{\phantom{=}}-\sqrt{\left(\frac{1}{m(i)}-\frac{2}{m(i-1)+m(i+1)}\right)^2+\left(\frac{m(i+1)-m(i-1)}{2m(i-1)m(i+1)}\right)^2}.
\end{align*}
Since there does not have the symmetry as in the Bakry-\'{E}mery Ricci flow on $C_3$, one may not be able to obtain properties similar to Proposition \ref{PreserveMaxAndMin}, see Figure \ref{FlowC4}. But a similar proof to Theorem \ref{MaximumDecreasingForC3} shows that for a.e. $t \in [0, T)$, $\frac{d}{dt} \max_{x \in V}m(t, x) \leq -\frac{2}{\max_{x \in V}m(t, x)}$. So that we have the following theorem.
\begin{theorem}
 For $C_4$ and any $m_0 \in \mathbb{R}_+^V$, let $m \in C^{\infty}([0, T) \times V, \mathbb{R}_+)$ be a solution to the Bakry-\'{E}mery Ricci flow. Then $T \leq \frac{\max_{x \in V}m_0(x)^2}{4}$.
\end{theorem}
\begin{remark}
 Same as the solution to the Bakry-\'{E}mery Ricci flow on $C_3$, $\min_{x \in V}m(t, x)$ of the solution to the Bakry-\'{E}mery Ricci flow does not have monotonically incresing properties, see Figure \ref{FlowC4}.
\end{remark}
\begin{figure}[H]
 \centering
  \includegraphics[width = 0.7\textwidth]{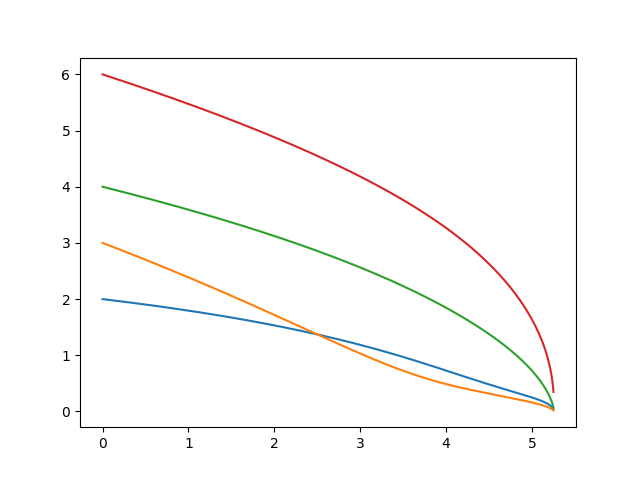}
 \caption{Bakry-\'{E}mery Ricci flow on $C_4$ with $m_0 = (2, 3, 4, 6)$. One can see that $\min_{x \in V}m(t, x)$ is decreasing. Moreover, for the vertex with the smallest initial weight, after a period of time, the weight is no longer the smallest.}
 \label{FlowC4}
\end{figure}

We now consider the Bakry-\'{E}mery Ricci flow on $C_k$ with $k \geq 5$. For $m: V \to \mathbb{R}_+$, the Bakry-\'{E}mery curvature is given by, for $i = 0, 1, 2, \cdots, k-1$,
\begin{align*}
 &\mathrel{\phantom{=}}\mathrm{Ric}_m(i) \\
 & = \frac12\left(\frac{1}{m(i-1)} +\frac{1}{m(i+1)}\right)-\frac12\sqrt{\left(\frac{1}{m(i-1)}-\frac{1}{m(i+1)}\right)^2+\frac{4}{m(i)^2}}.
\end{align*}
In fact, it is the minimal eigenvalue of a matrix $(\frac{1}{m(i-1)}, \frac{1}{m(i)}; \frac{1}{m(i)}, \frac{1}{m(i+1)})$.

\begin{proposition}\label{MinimumIncreasing}
 For $C_k$ with $k \geq 5$, $m_0: V \to \mathbb{R}_+$, let $m(t, \cdot)$ be the solution to the Bakry-\'{E}mery Ricci flow. Then $\min_{x \in V}m(t, x)$($\max_{x \in V}m(t, x)$ resp.) is non-decreasing(non-increasing resp.) in $t$.
\end{proposition}
\begin{proof}
 W.l.o.g., we prove the result for the minimum. Since the graph is finite, for a.e. $s \in [0, T)$, there exists $x_s = \mathrm{argmin}_{x \in V} m(s, x)$ such that 
 \[
  \frac{d}{dt}\Big|_{t=s}m(t, x_s) = \frac{d}{dt}\Big|_{t=s}\min\limits_{x \in V}m(t, x).
 \]
 Note that
 \begin{align*}
 &\mathrel{\phantom{=}}\mathrm{Ric}_m(i) \\
 &=2\left(\frac{1}{m(i-1)m(i+1)}-\frac{1}{m(i)^2}\right) \\
 &\mathrel{\phantom{=}} \times \left(\frac{1}{m(i-1)}+\frac{1}{m(i+1)}+\sqrt{\left(\frac{1}{m(i-1)}-\frac{1}{m(i+1)}\right)^2+\frac{4}{m(i)^2}}\right)^{-1}.
 \end{align*}
 We have 
 \[
  \mathrm{Ric}_{m(s, \cdot)}(x_s) \leq 0.
 \]
 Hence by the Ricci flow equation, $\frac{d}{dt}\Big|_{t = s}m(t, x_s) \geq 0$. This proves the result.
\end{proof}
\begin{proposition}
 For $C_k$ with $k \geq 5$, any $m_0: V \to \mathbb{R}_+$, there exists a solution $m \in C^{\infty}([0, \infty) \times V, \mathbb{R}_+)$ to the Bakry-\'{E}mery Ricci flow. Moreover, $m([0, \infty) \times V)$ is contained in a compact subset in $\mathbb{R}_+^V$.
\end{proposition}
\begin{proof}
 By Proposition \ref{MinimumIncreasing}, for any $t \in [0, T)$,
 \[
  \min_{y \in V}m_0(y) \leq m(t, x) \leq \max_{y \in V}m_0(y), \quad \forall \ x \in V.
 \]
 Hence, $m([0, \infty) \times V)$ is contained in a compact subset in $\mathbb{R}_+^V$. This yields that $T = \infty$.
\end{proof}
\begin{theorem}
 For $C_k$ with $k \geq 5$, the solution to the Bakry-\'{E}mery Ricci flow converges to a constant weight $c \in \mathbb{R}_+^V$, where $c>0$.
\end{theorem}
\begin{proof}
By Proposition \ref{MinimumIncreasing}, we may assume that $M(t):= \max_{x \in V}m(t, x) \to b$ and $m(t):= \min_{x \in V}m(t, x) \to a$ as $t \to \infty$. We only need to prove that $a = b$. Suppose that it is not true, i.e., $b > a$. We show that this yields a contradiction. 

One easily checks that 
\begin{align*}
 \sum_{i \in V} \partial_t m(t, i) &= -\sum_{i \in V} \mathrm{Ric}_{m(t, \cdot)}(i) \\
 & \geq \sum_{i \in V}\frac12\left(-\frac{1}{m(i-1)}-\frac{1}{m(i+1)}+\frac{2}{m(i)}\right) \geq 0.
\end{align*}
On the other hand, since $M(t)$ is a monotonically increasing Lipschitz function, we know that for a.e. $s \in [0, \infty)$, $\frac{d}{dt}\Big|_{t = s} M(t)$ exists and tends to $0$. For large $t$, we claim that there exists $\varepsilon_0 > 0$ such that $Q(t):= \sum_{i \in V} \mathrm{Ric}_{m(t, \cdot)} \leq -\varepsilon_0$. 

Suppose that is not true, i.e., for any $\varepsilon > 0$ and $T_1 >0$, there exist $t > T_1$ such that $Q(t) \geq -\varepsilon$. For large time $t$, we may assume that $\max_{x \in V}m(t, x) \leq \frac{3b}{2}$ and $\min_{x \in V}m(t, x) \geq \frac{a}{2}$. Note that for any $0 \leq \varepsilon \leq 1$ and $x \geq 0$, $\sqrt{x^2 + 1} \geq (1-\varepsilon)x + \sqrt{\varepsilon}$. Thus, we have
\begin{align*}
  -2\varepsilon &\leq 2Q(t) \\
  &\leq \sum_{i \in V}(\frac{1}{m(t, i-1)}+\frac{1}{m(t, i+1)} - (1-\varepsilon)\frac{2}{m(t, i)}) \\
  &\mathrel{\phantom{=}} -\sum_{i \in V} \sqrt{\varepsilon}\Big|\frac{1}{m(t, i-1)}-\frac{1}{m(t, i+1)}\Big| \\
  &= \sum_{i \in V}\left(\frac{2\varepsilon}{m(t, i)} - \sqrt{\varepsilon}\Big|\frac{1}{m(t, i-1)}-\frac{1}{m(t, i+1)}\Big|\right) \\
  &\leq \frac{4k\varepsilon}{a} - \sum_{i \in V}\sqrt{\varepsilon}\Big|\frac{1}{m(t, i-1)}-\frac{1}{m(t, i+1)}\Big|. 
\end{align*}
It follows that for all $i \in V$,
\[
 \Big|\frac{1}{m(t, i-1)}-\frac{1}{m(t, i+1)}\Big| \leq C_0\sqrt{\varepsilon}, 
\]
where $C_0 = \frac{4k}{a}+2$.
Now, if $k$ is odd, then for any $i, j \in V$, we have 
\[
 \Big|\frac{1}{m(t, i)} - \frac{1}{m(t, j)}\Big| \leq C_0k\sqrt{\varepsilon},
\]
and thus,
\[
 \Big|\frac{1}{M(t)} - \frac{1}{m(t)}\Big| \leq C_0k\sqrt{\varepsilon},
\]
which is impossible if we choose $\varepsilon$ small enough. If $k$ is even, we have
\[
 \Big|\frac{1}{m(t, i)} - \frac{1}{m(t, j)}\Big| \leq C_0k\sqrt{\varepsilon}, \quad \forall \ i, j \in \{0, 2, \cdots, k-2\},
\]
and
\[
 \Big|\frac{1}{m(t, i)} - \frac{1}{m(t, j)}\Big| \leq C_0k\sqrt{\varepsilon}, \quad \forall \ i, j \in \{1, 3, \cdots, k-1\}.
\]
Choose $\varepsilon < \left(\frac{b-a}{2C_0kab}\right)^2$, then
\[
 \Big|\frac{1}{m(t, i)} - \frac{1}{M(t)}\Big| \leq \frac12\left(\frac{1}{a}-\frac{1}{b}\right), \quad \forall \ i \in \{0, 2, \cdots, k-2\},
\]
and
\[
 \Big|\frac{1}{m(t, j)} - \frac{1}{m(t)}\Big| \leq \frac12\left(\frac{1}{a} - \frac{1}{b}\right), \quad \forall \ j \in \{1, 3, \cdots, k-1\}.
\]
It follows that for $x_t \in \mathrm{argmax}_{i \in V}m(t, i)$,
\begin{align*}
 \partial_tM(t) &\geq -\frac12\left(\frac{1}{m(t, x_t-1)}+ \frac{1}{m(t, x_t+1)}\right) + \frac{1}{m(t, x_t)} \\
 &\geq -\frac{1}{m(t)} - \frac12\left(\frac{1}{a}-\frac{1}{b}\right) + \frac{1}{m(t, x_t)} \\
 &\to \frac12\left(\frac{1}{a}-\frac{1}{b}\right),
\end{align*}
which contradicts the fact $\frac{d}{dt} M(t) \to 0$. Hence, the claim holds.

Due to the claim, we have for large $t$,
\[
 \sum_{i \in V} \partial_tm(t, i) \geq \varepsilon_0.
\]
It follows that $\sum_{i \in V}m(t, i) \to +\infty$ as $t \to \infty$. On the other hand, for large $t$,
\[
 \sum_{i \in V}m(t, i) \leq \frac{3k}{2}b.
\]
This is a contradiction, and implies that $b = a$, as desired.
\end{proof}
\begin{figure}[H]
 \centering
  \includegraphics[width = 0.7\textwidth]{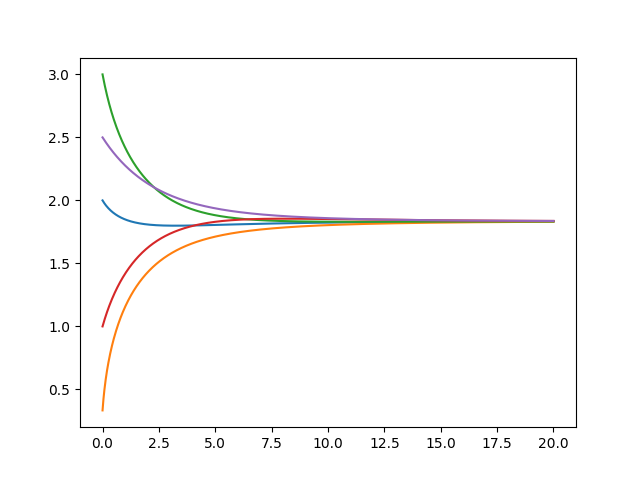}
 \caption{Bakry-\'{E}mery Ricci flow on $C_5$ with $m_0 = (2, \frac13, 3, 1, \frac52)$.}
\end{figure}

Finally, we briefly introduce the normalized Bakry-\'{E}mery on $C_3$ and $C_4$. For the normalized Ricci flow on $C_3$, if we solve the equation $\mathrm{Ric}_m(0) = \mathrm{Ric}_m(1) = \mathrm{Ric}_m(2)$ under the constraint $\sum_i \frac{1}{m(i)}=1$, one gets that $m(0) : m(1) : m(2) = 5:5:2$ or $m(0):m(1):m(2) = 3:3:4$ or $m=\text{const}$. Draw the phase diagram of the $m(t, 0)$ and $m(t, 1)$ under the constraint $\sum_i \frac{1}{m(t, i)} = 1$, as shown in the Figure \ref{PhaseDiagram}. One can see that the point $(3, 3)$ is unstable, which implies that the normalized Ricci flow on $C_3$ will not converge to a constant if it converge.
\begin{figure}[H]
 \centering
  \includegraphics[width = 0.5\textwidth]{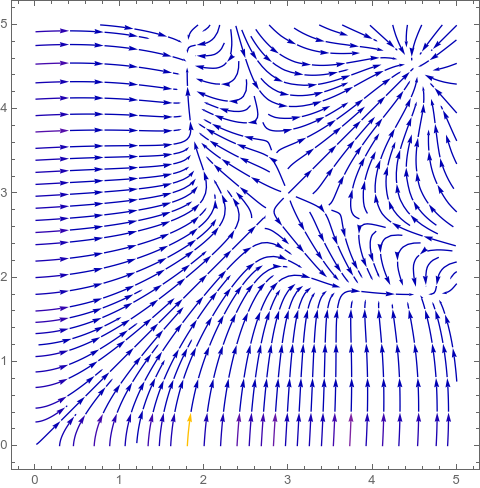}
 \caption{}
 \label{PhaseDiagram}
\end{figure}

For the normalized Bakry-\'{E}mery Ricci flow on $C_4$, if we solve the equation $\mathrm{Ric}_m(0) = \cdots \mathrm{Ric}_m(3)$, one gets that $m = \text{const}$, which implies that the normalized Ricci flow will converge to a constant if it converges. The numerical result suggests that it converges for any initial data.
\begin{figure}[H]
\centering
\includegraphics[width=0.5\textwidth]{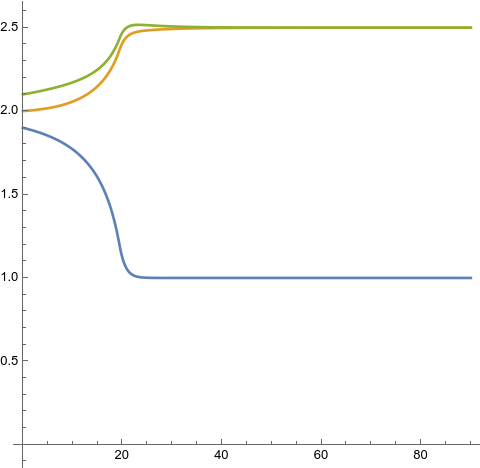}
\caption{Normalized Ricci flow on $C_3$ with $m_0 = (1.9, 2, 2.1)$}
\end{figure}
\begin{figure}[H]
\centering
\includegraphics[width=0.5\textwidth]{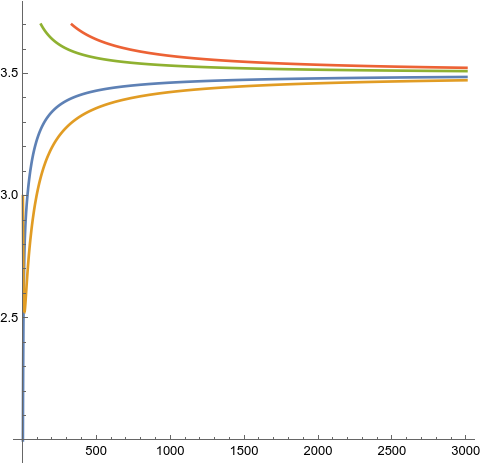}
\caption{Normalized Ricci flow on $C_4$ with $m_0 = (2, 3, 4, 5)$}
\end{figure}

\section*{Acknowledgments}
The authors would like to thank Florentin M\"{u}nch and Wanjun Ai for their helpful advice. B. Hua is supported by NSFC, no. 12371056, and by Shanghai Science and Technology Program[Project No. 22JC1400100]. Y. Lin is supported by NSFC, no. 12071245.

\bibliographystyle{alpha}
\bibliography{BakryEmeryRicciFlow}

\end{document}